\documentclass[a4paper,11pt]{amsart}

\usepackage[T1]{fontenc}
\usepackage[utf8]{inputenc}

\usepackage{amssymb}

\makeatletter
\@namedef{subjclassname@2020}{%
  \textup{2020} Mathematics Subject Classification}
\makeatother

\allowdisplaybreaks[2]

\newtheorem{theorem}{Theorem}[section]

\newtheorem{lemma}[theorem]{Lemma}
\newtheorem{corollary}[theorem]{Corollary}
\theoremstyle{definition}
\newtheorem{remark}[theorem]{Remark}

\newtheorem{example}[theorem]{Example}

\numberwithin{equation}{section}

\begin{document}

\title[Automorphisms and derivations]{Automorphisms and derivations of finite-dimensional  algebras}

\thanks{Supported by the Slovenian Research Agency (ARRS) Grant P1-0288. }

\author{Matej Bre\v sar} 
\address{Faculty of Mathematics and Physics,  University of Ljubljana,  and Faculty of Natural Sciences and Mathematics, University of Maribor, Slovenia}
\email{matej.bresar@fmf.uni-lj.si}



\begin{abstract} Let $A$ be a finite-dimensional algebra over a field $F$  with {\rm char}$(F)\ne 2$. We show that a linear map $D:A\to A$  satisfying  $xD(x)x\in [A,A]$ for every $x\in A$
is  the sum of an inner  derivation and a linear map whose  image lies in the radical of $A$. Assuming additionally that $A$ is semisimple and char$(F)\ne 3$, we show that 
 a linear map $T:A\to A$  satisfies  
$T(x)^3- x^3 \in [A,A]$   for every $x\in A$ if and only if 
 there exist a Jordan automorphism $J$ of $A$ lying in the multiplication algebra  of $A$ and a central element $\alpha$ 
  satisfying $\alpha^3=1$ such that  $T(x)=\alpha J(x)$ for all $x\in A$.
%if and only if there exists a central element $\beta$  such that $\beta^3=1$ and $\beta T$ is either an inner automorphism or an antiautomorphism belonging to the multiplication algebra of $A$.
% it is either  an inner automorphism or   a central multiple of an antiautomorphism belonging to the multiplication algebra of $A$, followed by a left multiplication by a central element whose cube is $1$.
%is a  either a central multiple of an inner automorphism or   a central multiple of an antiautomorphism belonging to the multiplication algebra of $A$.
These two results are applied to the study of local derivations and local (Jordan) automorphisms. In particular, the second result is used to prove that every
local Jordan automorphism of a finite-dimensional simple algebra $A$ (over a field $F$ with char$(F)\ne 2,3$) is a Jordan automorphism.
\end{abstract}

\newcommand\E{\ell}
\newcommand\mathcalM{{\mathcal M}}
\newcommand\pc{\mathfrak{c}}

\newcommand{\enp}{\begin{flushright} $\Box$ \end{flushright}}

\maketitle

\section{Introduction} 

The theory of functional identities deals with the description of functions on rings and algebras  that satisfy certain identities  \cite{FIbook}. 
In this paper, we consider a more general type of  problems where expressions involving functions on an algebra $A$ are, instead of being always 0 as is usually the case with functional identities, contained in a relatively large subset of $A$, namely in $[A,A]$, the linear span of all commutators in $A$. In the special case where $A$ is the matrix algebra $M_n(F)$ this can be equivalently stated as that the trace of these expressions is always zero. Therefore, the relations that we will study are, in some sense, also more general than trace identities (see, e.g., \cite{AGPR}). We will not, however, develop some general theory, but consider only two special cases that merely indicate a possible new approach to ``generalized identities'' in rings and algebras. The two new type theorems  will be shown to have
 applications to a
%... two theorems present, on the one hand, a new type of results that extends both functional identities and trace identities, and, on the other hand, have  applications to
 well-studied research topic, i.e., to the theory of local derivations and local automorphisms. 

Let us be more specific. All our results consider  a finite-dimensional algebra $A$ over a field $F$ with char$(F)\ne 2$. In Section \ref{s3}, we additionally assume that
char$(F)\ne 3$. % (examples show that these characteristic assumptions are necessary). 
Section \ref{s2} is centered around the condition that a linear map $D:A\to A$ satisfies
 \begin{equation} \label{2d} xD(x)x\in [A,A]\quad\mbox{for all $x\in A$},\end{equation}
and Section \ref{s3} is centered around the condition  that a linear map $T:A\to A$ satisfies
 \begin{equation} \label{3a} T(x)^3  - x^3\in [A,A]\quad\mbox{for all $x\in A$.}\end{equation}
It is immediate that inner derivations satisfy \eqref{2d} and inner automorphisms satisfy \eqref{3a}.

Our first main result, Theorem \ref{d}, states that \eqref{2d} implies that 
 $D$ is the sum of an inner derivation and a linear map having the image in rad$(A)$, the radical of $A$. Although maps with the image in rad$(A)$ do not always satisfy \eqref{2d}, it is still  reasonable that they appear in the conclusion since rad$(A)$ is sometimes contained in $[A,A]$.
 
  Condition \eqref{3a} can be studied similarly as condition   \eqref{2d}, but the results are more involved. % We will restrict ourselves to the case where $A$ is a semisimple algebra.
 Theorem \ref{a}, which is our second main result,  states that if $A$ is semisimple then
condition  \eqref{3a} is equivalent to the condition that  there exist  a central element $\alpha$ satisfying $\alpha ^3=1$ and a Jordan automorphism  $J$ of $A$
belonging to the multiplication algebra of $A$ such that $T(x)=\alpha J(x)$ for all $x\in A$. In Corollaries \ref{ac2} and \ref{ac3}, we consider the situation where $A$ is a general, not necessarily semisimple finite-dimensional algebra. However, these results are not as definitive as Theorem \ref{d}.
 
 %The analogous result on condition \eqref{3a}, Theorem \ref{a}, has a slightly more  involved statement: a linear map $T:A\to A$
%satisfies  $T(x)^3- x^3 \in [A,A]$ for all $x\in A$ if and only if 
 %there exists an $\beta$ in the center $Z$ of $A$ such that $\beta^3=1$ and $\beta T$ is either an inner automorphism or an antiautomorphism belonging to the multiplication algebra  $M(A)$ of $A$.
The proofs of Theorems \ref{d} and \ref{a} use the classical theory of finite-dimensional algebras together with some results on Jordan maps. We also provide several examples that justify  the assumptions.

As already indicated, these two theorems  are applicable to the study of  local derivations and local automorphisms. 
 A {\em local derivation} of an algebra $A$ is a linear map $D:A\to A$ with the property that for each $x\in A$, there is a derivation $D_x:A\to A$ such that
$D(x)=D_x(x)$. This notion was introduced in 1990 by Kadison
\cite{K} and independently by Larson and Sourour \cite{LS} who also introduced {\em local automorphisms}.  These are defined analogously, i.e., as linear maps $T:A\to A$ such that for each $x\in A$, there is an automorphism $T_x:A\to A$ satisfying
$T(x)=T_x(x)$. The definitions of other types of local maps should now be self-explanatory. The standard question is whether local derivations, local automorphisms, etc.\ are derivations, automorphisms, etc.\ 
Over the last three decades,  positive answers were obtained in various algebras $A$ (occurring not only in algebra but also if not primarily in functional analysis).  We refer   to a few recent publications \cite{zpdbook,Co,GM} which contain some
historical remarks and further references.

Theorem \ref{d} immediately implies that every local inner derivation of  a finite-dimensional semisimple algebra $A$  (over a field $F$ with char$(F)\ne 2$)  is an inner derivation (Corollary \ref{cd2}). In Example \ref{ede} we show that a similar conclusion for general derivations does not always hold. The automorphism case is more complex and interesting.
As will be explained in Section \ref{s3},  local Jordan automorphisms are more natural in finite dimensions than local automorphisms. Using Theorem \ref{a}, we will show that every local Jordan automorphism of a
 finite-dimensional simple algebra $A$ over a field $F$ with char$(F)\ne 2,3$
 is a Jordan automorphism (Theorem \ref{a2}). This is the third main result of this paper.

To the best of our knowledge, our results on local maps are new and cover a basic class of algebras which is quite different from those treated by other authors.
Papers on local maps  are often based 
on the existence of some special elements like idempotents. The class of  simple algebras, however, includes division algebras which
  contain no such elements.

\section{Derivations}\label{s2}

 %Unless otherwise stated, $A$ will stand for a finite-dimensional simple algebra over a field $F$. We are not assuming that $A$ is central, that is, the center $Z$ of $A$ may contain $F$ as its proper subfield. Of course, we may also regard $A$ as a central simple algebra over the field  $Z$.

 We start with a simple but important lemma. As usual, we write $[x,y]$ for the commutator $xy-yx$, and $[A,A]$ for the linear span of all $[x,y]$, $x,y\in A$.

\begin{lemma}\label{l1}
Let $A$ be a finite-dimensional simple algebra. If $c\in A$ is such that $cA\subseteq [A,A]$, then $c=0$. 
\end{lemma}

\begin{proof} From $xcy = [x,cy] + cyx \in [A,A]$ we see that  the  ideal of $A$ generated by $c$ is contained in $[A,A]$. However, it
 is easy to see that $[A,A]$ is a proper subspace of $A$ (in fact, it has codimension $1$ if viewed as a space over the center $Z$ of $A$ \cite[Exercise 4.12]{INCA}). Hence, $c=0$.
\end{proof}

\begin{remark}\label{l1a}
 We will also need the following  technical variation of Lemma \ref{l1}: If char$(F)\ne 2$ and $A$ is as in the lemma, then $cx^2 \in [A,A]$ for every $x\in A$ implies $c=0$. This follows immediately from $x=\frac{1}{2}((x+1) ^2- x^2-1^2)$ (here we used that a finite-dimensional simple algebra always contains $1$). Note that  we may replace the condition $cx^2 \in [A,A]$ by $xcx\in [A,A]$ since $xcx = cx^2 + [x,cx]$.
 %Let  $M$ be a maximal ideal of  a finite-dimensional algebra  $A$. If $c\in A$ is such that $cx^2\subseteq [A,A]$ for all $x\in A$, then $c\in M$. 
\end{remark}

Let $D$ be a  linear map from an algebra $A$ to itself. 
%given a linear map $D:A\to A$, w
 A linear map $\Delta:A\to A$ is called a  {\em Jordan $(D,D)$-derivation} if it satisfies $$\Delta(x^2) = D(x)x+xD(x)\quad\mbox{for all $x\in A$}.$$
  This notion was introduced in \cite{Bj} (as a special case of  more general Jordan $(D,G)$-derivations)  in order to study the classical Jordan derivations  (the case where $\Delta = D$) on tensor products.  Somewhat to the author's surprise, Jordan $(D,D)$-derivations
 naturally occur in the proof of the next theorem, and the result from \cite{Bj} stating that they satisfy $\Delta(xy)=D(x)y+xD(y)$ provided that $A$ is a  semiprime algebra (over a field of characteristic not $2$) is applicable. 
 
 Let us also recall a few standard definitions and facts.  The {\em radical} of a finite-dimensional algebra $A$, denoted rad$(A)$, is the unique maximal nilpotent ideal of $A$. An equivalent description is that rad$(A)$ is the intersection of all maximal ideals of $A$.
 If rad$(A) =\{0\}$, then $A$ is a semisimple algebra, i.e., $A$ is a direct sum of ideals each of which is a simple algebra. The quotient algebra $A/{\rm rad}(A)$ is always semisimple.
 %We must recall one more definition to state our first theorem.
 Finally, the 
  algebra of all linear maps from $A$ to $A$ of the form $x\mapsto \sum_i a_ixb_i$ for some $a_i,b_i\in A$ is called the 
 {\em multiplication algebra} of $A$. It will be denoted by $M(A)$.
 
  We can now state our first main theorem.

\begin{theorem}\label{d}
Let $A$ be a finite-dimensional algebra over a field $F$ with {\rm char}$(F)\ne 2$. If a linear map $D:A\to A$ satisfies 
 $xD(x)x\in [A,A]$ for every $x\in A$, then $D$ is the sum of an inner  derivation of $A$ and a linear map from $A$ to {\rm rad}$(A)$.
\end{theorem}

\begin{proof}We write $x\equiv y$ for $x-y\in [A,A]$. 
Our assumption is  thus $xD(x)x\equiv 0$ for all $x\in A$.
%Since $D(x) x^2 = [D(x)x,x] + xD(x)x$, we have $D(x) x^2 \in [A,A]$.  
As {\rm char}$(F)\ne 2$, replacing 
$x$ by $x\pm y$   implies
 \begin{equation}\label{d1} yD(x)x + xD(y)x + xD(x)y \equiv 0\quad\mbox{for all $x,y \in A$.}
\end{equation}
This will be our basic relation in the course of the proof.

Let $M$ be a maximal ideal of $A$. Taking $y\in M$ it follows from \eqref{d1} that $xD(y)x\in M+ [A,A]$ for every $x\in A$. Hence, $c=D(y)+M\in A/M$ satisfies $ucu\in [A/M,A/M]$ for every $u\in A/M$. 
As $A/M$ is simple, Remark \ref{l1a} tells us that
$c =0$, i.e., $D(y)\in M$. We have thus proved that $D(M)\subseteq M$ for every maximal ideal $M$ of $A$.

As rad$(A)$ is the intersection of all maximal ideals of $A$, it follows that $D({\rm rad}(A))\subseteq {\rm rad}(A)$. We can thus define $\overline{D}:A/{\rm rad}(A)\to A/{\rm rad}(A)$ by
$$\overline{D}(x+ {\rm rad}(A)) = D(x)+ {\rm rad}(A).$$
Note that $\overline{D}$ is a linear map satisfying $$v\overline{D}(v)v\in [A/{\rm rad}(A),A/{\rm rad}(A)]\quad\mbox{for all $v\in A/{\rm rad}(A)$}.$$
Assuming that the theorem is true for semisimple algebras, it follows from this relation that $\overline{D}$ is an inner derivation of $A/{\rm rad}(A)$, which further implies that $D$ is of the desired form. Therefore, we may assume without loss of generality that $A$ is a semisimple algebra. 

Thus, $A=A_1\oplus \dots\oplus A_r$ where each $A_i$ is a simple algebra. Since $A_i$ is the intersection of the maximal ideals  $A_1\oplus\dots\oplus A_{j-1}\oplus A_{j+1} \dots\oplus A_r$ with $j\ne i$,  it is invariant under 
$D$ by what we proved above. Considering the restriction of $D$ to each $A_i$, we see that there is no loss of generality in assuming that $A$ is a simple algebra.

Let $Z$ denote the center of $A$.  Take $z\in Z$. Substituting $z y$ for $y$ in \eqref{d1} we obtain
$$z yD(x)x + xD(z  y)x + z  xD(x)y\equiv 0\quad\mbox{for all $x,y \in A$.}$$
On the other hand, since $z[A,A] \subseteq [A,A]$,
  \eqref{d1} shows that 
 $$ z yD(x)x + z xD( y)x + z  xD(x)y  \equiv 0\quad\mbox{for all $x,y \in A$.}$$
 Comparing both relations we obtain
$$x\big(D(z y) - z D(y)\big)x \equiv 0 \quad\mbox{
 for all $x,y \in A$.}$$
 %Replacing $x$ by $x+1$ gives
 %$$x\big(D(z y) - z D(y)\big) + \big(D(z y) - z D(y)\big) x\equiv 0,$$ which along with the obvious relation
 %$$x\big(D(z y) - z D(y)\big) \equiv \big(D(z y) -z D(y)\big) x$$
%yields
 % $$\big(D(z y) - z D(y)\big)x \equiv 0\quad\mbox{for all $x,y \in A$.}$$
  Since $A$ is simple, we see from Remark \ref{l1a}  that $D(z y) = z D(y)$, i.e., $D$ is $Z$-linear. But then $ D$ belongs to the multiplication algebra $M(A)$ \cite[Lemma 1.25]{INCA}. 
  
  Let $a_i,b_i\in A$ be such that $$D(x)=\sum_i a_ixb_i\quad\mbox{for all $x\in A$}.$$  We  have
  $$xD(y)x = \sum_i x a_iyb_ix  = \sum_i b_ix ^2  a_iy +  \sum_i [xa_iy,b_ix ]$$
  and so   $xD(y)x  \equiv \sum_i b_ix ^2  a_iy$.
 Using also
 $yD(x)x \equiv D(x)xy $
 we now see that \eqref{d1} can be rewritten as 
 $$\big(D(x) x +  \sum_i b_ix ^2  a_i + xD(x)\big) y\equiv 0\quad\mbox{for all $x,y \in A$.}$$
Therefore, by  Lemma \ref{l1},  $$D(x) x +  \sum_i b_ix ^2  a_i + xD(x) =0\quad\mbox{for all $x \in A$.}$$
  This means that the linear map $\Delta:A\to A$ defined by
  $$\Delta(x)=-\sum_i b_i x a_i$$
is a Jordan $(D,D)$-derivation. By \cite[Theorem 4.3]{Bj}, $\Delta$  satisfies 
    \begin{equation}\label{d2}
\Delta(xy) = D(x)y+xD(y)\quad\mbox{for all $x,y \in A$.}\end{equation}
    Writing first $1$ for $x$ and then $1$ for $y$ we see that $\beta = D(1)$ lies in  $Z$. From
    $\Delta(x)= D(x) + \beta x$ and \eqref{d2} it follows that $d:A\to A$ defined by $$d(x) = D(x) - \beta x$$
    is a derivation. Moreover, $d$ is $Z$-linear. It is a standard fact  that such a derivation is inner (see, e.g., \cite[Exercise 4.20]{INCA}).  Thus, there is an $a\in A$ such that  $D(x) = \beta x + [a,x]$ for all $x\in A$.

    The proof will be complete by showing that $\beta =0$.
     Since 
    \begin{equation}\label{nin}x[a,x]x= [xax,x]\equiv 0\end{equation}  it follows from $xD(x)x\equiv 0$ that $\beta x^3 \equiv 0$. If $\beta$ was not $0$, every
    $x\in A$ would satisfy  \begin{equation}\label{d'}
    x^3 \equiv 0.\end{equation}  To show that this  is not  true, first note that a complete linearization of \eqref{d'} gives 
\begin{equation} x_1x_2x_3 + x_1x_3x_2 + x_2x_1x_3 + x_2x_3x_1+ x_3x_1x_2 + x_3x_2x_1\equiv 0 \label{d3}\end{equation}
for all $x_1,x_2,x_3 \in A$.
Since char$(F)\ne 2$, we also have
 \begin{equation} \label{d4}x_1^2 x_2 + x_1x_2 x_1 + x_2x_1^2 \equiv 0 \end{equation}
 for all $x_1,x_2 \in A$.
    Let $K$ be the algebraic closure of $Z$ and let
 $A_K=K\otimes A$ be the scalar extension of $A$ to $K$.   Take $y=\sum_j k_j\otimes x_j \in A_K$. Observe that $y^3$ is a sum of terms of the form
 $$k\otimes x _i^3,$$ $$k' \otimes (x_i^2 x_j + x_ix_j x_i + x_jx_i^2),$$
 and $$k''\otimes  ( x_ix_jx_k + x_ix_kx_j + x_jx_ix_k + x_jx_kx_i+ x_kx_ix_j + x_kx_jx_i).$$
 Using \eqref{d'}, \eqref{d3}, and \eqref{d4} it follows that $y ^3 \in [A_K,A_K]$.  However, $A_K\cong M_n(K)$ and so this obviously cannot hold for every $y\in A_K$ (e.g., for an idempotent of rank $1$). This contradiction  proves that  $\beta =0$.
\end{proof}

\begin{remark}\label{remder}
The last paragraph of the proof could be  shortened if char$(F)$ was different from $3$. Indeed, since $x_1^2 x_2 \equiv x_1x_2 x_1 \equiv x_2x_1^2 $, in this case it is enough to  apply Lemma \ref{l1} to \eqref{d4}. 
If, however,  char$(F)= 3$, then    \eqref{d4} holds for any algebra over $F$ as  $x_1^2 x_2 + x_1x_2 x_1 + x_2x_1^2$ is then equal to $[x_1,[x_1,x_2]]$. \end{remark}

As noticed in \eqref{nin}, every inner derivation $D$ satisfies  $xD(x)x\in [A,A]$.  
Thus, if the algebra $A$ is such that rad$(A)\subseteq [A,A]$, then Theorem \ref{d} turns into an ``if and only if'' theorem. A simple concrete example is the algebra $A$ of all upper triangular matrices over $F$ (which  actually satisfies  rad$(A)= [A,A]$). Another example is of course any semisimple algebra. We record this as a corollary.

\begin{corollary}\label{dd}
Let $A$ be a finite-dimensional semisimple algebra over a field $F$ with {\rm char}$(F)\ne 2$. The following conditions are equivalent for a linear map $D:A\to A$:
\begin{enumerate}
\item[{\rm (i)}]  $xD(x)x\in [A,A]$ for every $x\in A$.
\item[{\rm (ii)}]
 $D$ is an inner derivation.
 \end{enumerate}
\end{corollary}

%\begin{proof}
%The implication (i)$\implies$(ii) follows from Theorem \ref{d}. The implication  (ii)$\implies$(i)  is trivial  since $x[a,x]x= [xax,x]$.
%\end{proof}

Inner derivations $D$  also satisfy a simpler condition $xD(x)\in [A,A]$ for all $x\in A$.  However, so do many other maps, as the next example shows. There are thus good reasons for  considering the condition  $xD(x)x \in [A,A]$.

\begin{example}\label{rd} % However, so does, for example, a
Every map of the form $D(x)=\sum_i a_ixb_i-b_ixa_i$  satisfies $xD(x)\in [A,A]$ for all $x\in A$. This follows from  $x(a_ixb_i-b_ixa_i)=[xa_i,xb_i]$.
\end{example}

The following example shows that the assumption that char$(F)\ne 2$ is necessary in Corollary \ref{dd}.

\begin{example}
Let 
 $F=\mathbb F_2$ be the field  with 2 elements and let $A=M_2(F)$. Define $D:A\to A$ by
 $$D\left( \left[ \begin{matrix} x_{11} & x_{12} \cr x_{21} & x_{22} \cr \end{matrix} \right]\right) = \left[ \begin{matrix}x_{22} & x_{12} \cr 0 & x_{11} \cr \end{matrix}\right].$$
 Using $xy(x+y)=0$ for all $x,y\in F$ one can check that the trace of the matrix $xD(x)x$ is $0$ for every $x\in A$. Therefore,  $xD(x)x$ lies in  $[A,A]$. However, $D$ is not a derivation.
\end{example}

In the rest of this section we consider local (inner) derivations. %More precisely, we will consider local inner derivations.

%\begin{corollary}\label{cd}
%Let $A$ be a finite-dimensional algebra over a field $F$ with {\rm char}$(F)\ne 2$. Then every local inner derivation
%$D:A\to A$ is the sum of an inner  derivation and a linear map whose  image lies in {\rm rad}$(A)$.
%\end{corollary}

%Local derivations that are not derivations and have the image in rad$(A)$ indeed exist.

%\begin{example}Let $A$ be the $3$-dimensional unital algebra generated by an element $a$ such that $a^3=0$. Then the linear map $D:A\to A$ given by $D(1)=D(a^2) =0$ and $D(a)=a$ is a local derivation, but not 
%a derivation---see \cite[Example 8.9]{zpdbook} for details. \end{example}

\begin{corollary}\label{cd2}
Let $A$ be a finite-dimensional semisimple algebra over a field $F$ with {\rm char}$(F)\ne 2$. Then every local inner derivation
$D:A\to A$ is an inner  derivation.
\end{corollary}

\begin{proof}
The condition that $D$ is a local inner derivations means that for every $x\in A$, there exists an $a_x\in A$ such that $D(x)=[a_x,x]$. This obviously implies
$xD(x)x = [xa_x x,x]\in [A,A]$. Therefore, $D$ is an inner derivation by Corollary \ref{dd}.
\end{proof}

The next example shows that  Corollary \ref{cd2} cannot be extended to general, not necessarily inner derivations. In fact, this fails to hold even when $A$ is a field.  Of course, this can occur only if $A$ is an  inseparable field extension of $F$ in order to have nontrivial derivations. We remark that our  example is similar to the one from \cite{K}  which concerns the algebra $\mathbb C(X)$ (which, however, is infinite-dimensional over $\mathbb C$).

%\begin{example}
%Let $p$ be an odd prime. Consider the field of rational functions $A=\mathbb Z_p(X)$ as a finite extension of the field $F=\mathbb Z_p(X ^p)$. Let $f'$ be the usual derivative of $f\in A$. We remark that  $f'=0$ if and only if $f\in F$. 
 %Observe that, for any $g\in A$, $f\mapsto gf'$ is an $F$-linear derivation of $A$. This readily implies that every $F$-linear map of $A$ that sends $1$ to $0$ is a local derivation. However, such a map is not necessarily a derivation. For example, if it sends $X$ to $1$ and $X^2$ to $0$, then it certainly is not.\end{example}

\begin{example}\label{ede}
Let $p$ be an odd prime and let $F=\mathbb F_p(t)$ be the  rational function field over the field with $p$ elements $\mathbb F_p$. If  $\alpha$ is a root of the (irreducible and inseparable) polynomial $X ^p - t$, then  $A=F(\alpha)$ has degree  $p$ over $F$.
For any $f= \sum_{k=0}^{p-1} a_k \alpha ^k \in A$, $a_k\in F$, define $f'=  \sum_{k=0}^{p-1} k a_k \alpha ^{k-1}$. Observe that $f\mapsto f'$ is an $F$-linear derivation of $A$ whose kernel is $F$.  For
any $g \in A$, $f\mapsto gf'$ 
is also an $F$-linear derivation of $A$. This readily implies
that every $F$-linear map of $A$ that sends $1$ to $0$ is a local derivation. However,
such a map is not necessarily a derivation. For example, if it sends $\alpha$ to $1$
and $\alpha^2$
to $0$, then it certainly is not.
\end{example}

Corollary \ref{cd2} also does not hold without the assumption of finite dimensionality.

\begin{example} There exist (infinite-dimensional) division algebras $D$ in which every 
nonzero inner derivation is surjective \cite{C}. Every linear map  from $D$ to $D$ that vanishes at central elements is then a local inner derivation. However, such a map does not to be an inner derivation.
\end{example}

\section{Automorphisms, antiautomorphisms, and Jordan automorphisms}\label{s3}

Let $A$ be an algebra over a field $F$.
Recall that a {\em Jordan automorphism} of $A$ is a bijective linear map $J:A\to A$ satisfying
$$J(xy+yx)=J(x)J(y)+J(y)J(x)\quad\mbox{for all $x,y \in A$.}$$  If char$(F)\ne 2$, this condition  is equivalent 
to
 $J(x^2)=J(x)^2$ for all $x\in A$.
 Obvious examples of Jordan automorphisms are automorphisms and antiautomorphisms. These obvious examples are also the only examples if $A$ is a simple algebra
over a field $F$ with char$(F)\ne 2$. This is a special case of  the classical theorem of Herstein \cite{H} (the assumption from \cite{H} that the 
characteristic is not $3$ was later removed).
% Thus, in algebras of our interest, Jordan automorphisms are exactly (anti)automorphisms.

Our first lemma in this section essentially concerns central simple algebras, but for
 notational consistency 
we state it in a slightly different form.
% somewhat lengthier form than necessary.

\begin{lemma}\label{l2}
Let $A$ be a finite-dimensional simple algebra with center $Z$. If $J$ is a $Z$-linear Jordan automorphism of $A$, then $J(x)-x \in [A,A]$ for all $x\in A$.
\end{lemma}

\begin{proof} %To have a consistent notation,  we denote the the underlying field by $Z$. 
Let $K$ be the algebraic closure of $Z$ and let $A_K=K\otimes A$ be the  scalar extension  of $A$ to $K$. Then  $\overline{J} = {\rm id}_K\otimes J$ is a Jordan automorphism, and hence an  automorphism or an antiautomorphism of $A_K\cong M_n(K)$. If $\overline{J}(u) - u\in [A_K,A_K]$ for every $u\in A_K$, then for $u=1\otimes x$ we  obtain that 
$1\otimes (J(x) - x)$ is equal to an element of the form $\sum_i k_i\otimes [x_i,y_i]$, which  implies that $J(x)-x\in [A,A]$. Therefore, it is enough to consider the case where $A$ is the $K$-algebra $M_n(K)$.

If $J$ is an automorphism, then $J$ is inner
by the Skolem-Noether Theorem, so there is an invertible  $a\in A$ such that $$J(x)-x =axa^{-1}-x= [a,xa^{-1}]\in [A,A]$$
for every $x\in A$. Assume  that $J$ is an antiautomorphism.   
Then $x\mapsto J(x^t)$, where $x^t$ is the transpose of $x$, is an automorphism of $A$. Hence, $J(x^t) =axa^{-1}$ for all $x\in A$, or written equivalently,  $J(x)=ax^ta^{-1}$.
Since $ax^ta^{-1}$ and $x$ have the same trace it follows that $J(x)-x\in [A,A]$.
\end{proof}

We now begin our study of condition  $T(x)^3- x^3 \in [A,A]$. 

\begin{lemma}\label{ldo}
Let $A$ be an algebra over a field $F$ with {\rm char}$(F)\ne 2,3$. If a linear map $T:A\to A$ satisfies  $T(x)^3- x^3 \in [A,A]$ for every $x\in A$, then 
 \begin{equation}\label{h1}T(x)^2 T(y) - x^2y\in[A,A]\quad\mbox{for all $x,y \in A$.}
\end{equation}
Moreover, if  $A$ is finite-dimensional  and $\ker T\cap {\rm rad}(A)=\{0\}$, then $T$ is bijective and leaves every maximal ideal of $A$ invariant.
\end{lemma}

\begin{proof}
 As in the preceding section, we  write $x\equiv y$ for $x-y\in [A,A]$.

Since char$(F)\ne 2$, replacing $x$ by $x\pm y$ in $T(x)^3\equiv x^3$ gives
$$T(x)^2 T(y) + T(x)T(y)T(x) + T(y)T(x)^2 \equiv  x^2 y + xyx + yx^2.$$
As $T(x)^2 T(y) \equiv T(x)T(y)T(x) \equiv T(y)T(x)^2$,  $x^2 y \equiv xyx \equiv yx^2$, and {\rm char}$(F)\ne 3$,  \eqref{h1} follows.

Now assume that  $A$ is finite-dimensional  and $\ker T\cap {\rm rad}(A)=\{0\}$.
Take $y\in\ker T$. From \eqref{h1} we see that $x^2 y\in [A,A]$ for all $x\in A$.  Hence, for any ideal $M$ of $A$ we have  $u^2(y+M)\in [A/M,A/M]$ for all $u\in A/M$. Assuming that $M$ is maximal it follows from
Remark \ref{l1a}  that $y\in M$. This implies that $y\in {\rm rad}(A)$ and so $y=0$ by our assumption. %since we assumed that $\ker T\cap {\rm rad}(A)=\{0\}$. 
Thus, $T$ is bijective.

Finally, from \eqref{h1} it now follows that $w^2 T(y)\in M+[A,A]$ for every ideal $M$ of $A$, $y\in M$, and $w\in A$. As in the preceding paragraph we see that this implies $T(y)\in M$ if $M$ is maximal.
\end{proof}

We are ready to prove our second main result.

\begin{theorem}\label{a}
Let $A$ be a finite-dimensional semisimple algebra over a field $F$ with {\rm char}$(F)\ne 2,3$. The following conditions are equivalent for a linear map $T:A\to A$:
\begin{enumerate}
\item[{\rm (i)}] $T(x)^3- x^3 \in [A,A]$ for every $x\in A$.
\item[{\rm (ii)}]  
 There exist a Jordan automorphism $J$ of $A$ and an element $\alpha$ 
 from the center $Z$ of $A$ such that  $T(x)=\alpha J(x)$ for all $x\in A$. Moreover,  $J$ belongs to
   the multiplication algebra  $M(A)$ and  $\alpha$ satisfies
 $\alpha^3=1$. 
  \end{enumerate}
\end{theorem}

\begin{proof} 
(i)$\implies$(ii). Lemma \ref{ldo} tells us that $T$ is bijective and leaves maximal ideals of $A$ invariant. The latter implies that every simple component of $A$ is also invariant under $T$. 
 We may therefore assume   without loss of generality that  $A$ is a simple algebra.
% Replacing $x$ by $x\pm y$ in $T(x)^3\equiv x^3$ gives
%$$T(x)^2 T(y) + T(x)T(y)T(x) + T(y)T(x)^2 \equiv  x^2 y + xyx + yx^2.$$
%Since $T(x)^2 T(y) \equiv T(x)T(y)T(x) \equiv T(y)T(x)^2$,  $x^2 y \equiv xyx \equiv yx^2$, and {\rm char}$(F)\ne 3$, this can be written more simply as 
 %\begin{equation}\label{h1}T(x)^2 T(y) \equiv x^2y\quad\mbox{for all $x,y \in A$.} \end{equation}

%Suppose $y\in A$ is such that  $T(y)=0$. Then \eqref{h1} shows that $x^2y\equiv 0$ for every $x\in A$. %Replacing $x$ by $x+1$ it readily follows that  $xy\equiv 0$ for every $x\in A$, and hence $y=0$ by
%Observe that Lemma \ref{l1} and Remark \ref{l1a} trivially extend to semisimple algebras, so it follows that $y=0$.
  %This proves that $T$ is bijective.
  
 % As we know, $A=A_1\oplus \dots\oplus A_r$ where each $A_i$ is a simple algebra. Take $a_i\in A_i$ and write $T(a_i)=b_1+\dots + b_r$ where $b_j\in A_j$. Since $T$ is surjective, \eqref{h1} shows that
  %$z^2T(a_i) \in A_i + [A,A]$ for every $z\in A$. For each $j\ge 2$, we thus have $z_j^2 b_j\in [A_j,A_j]$ for every $z_j\in A_j$, and hence $b_j=0$ by Remark \ref{l1a}. We have thereby proved that $A_i$ is invariant under $T$. Accordingly,
  %without loss of generality we may assume that  $A$ is a simple algebra.
 
 Take $z$ from the center $Z$ of $A$. Replacing $y$ by $z y$ in \eqref{h1} we obtain $$T(x)^2 T(z y) \equiv z x^2 y\equiv T(x)^2 z T(y).$$
Since $T$ is surjective, we thus have $u^2 \big(T(z y)- z T(y)\big)\equiv 0$ for every $u\in A$.  
Hence, $T(z y)= z T(y)$ by Remark
 \ref{l1a}. This implies that  $T\in M(A)$  \cite[Lemma 1.25]{INCA}, that is,  there exist  $a_i,b_i\in A$ such that $$T(x)=\sum_i a_ixb_i\quad\mbox{for all $x\in A$}.$$ Hence,
$$T(x)^2 T(y) = \sum_i T(x)^2a_i y b_i \equiv \sum_i b_iT(x)^2a_i y,$$
which along with \eqref{h1} gives $$\Big( \sum_i b_iT(x)^2a_i - x^2\Big)y\equiv 0\quad\mbox{for all $x,y \in A$.}$$
Therefore, by Lemma \ref{l1}, 
\begin{equation}\label{h2}W(T(x)^2) = x^2\quad\mbox{for all $x\in A$,}\end{equation}
where   $$W(x)=\sum_i  b_i x a_i.$$
Substituting $x+1$ for $x$ in \eqref{h2} it follows that 
$$W(T(x)T(1) + T(1)T(x))= 2x.$$ This shows that $W$ is invertible and that  $S=W^{-1}$ satisfies  \begin{equation}\label{h300}2S(x)= T(x)T(1)+ T(1)T(x)  \quad\mbox{for all $x\in A$.}\end{equation} 
 By \eqref{h2},
$S(x^2)=T(x)^2 $ for all $x\in A$, and hence
%\begin{equation}\label{h3}S(x)= T(x)T(1)+ T(1)T(x)\quad\mbox{for all $x\in A$.}\end{equation}
%Moreover, by \eqref{h2},  
\begin{equation}\label{h30}S(xy +yx)=T(x)T(y)+T(y)T(x) \quad\mbox{for all $x,y\in A$.}\end{equation}
   As above, we denote by $K$ the algebraic closure of $Z$ and by
 $A_K = K\otimes A$ the scalar extension  of $A$ to $K$. 
 Observe that $\overline{S} = {\rm id}_K\otimes S$ and $\overline{T} = {\rm id}_K\otimes T$ are $K$-linear maps of $A_K\cong M_n(K)$ satisfying 
\begin{equation}\label{h3}\overline{S}(xy +yx)=\overline{T}(x)\overline{T}(y)+\overline{T}(y)\overline{T}(x) \quad\mbox{for all $x,y\in A_K$.}\end{equation}
Take an idempotent $e\in A_K$. From \eqref{h3} we see that $\overline{S}(e) = \overline{T} (e) ^2$ and $2\overline{S}(e) = \overline{T}(e) u + u\overline{T}(e)$
where $u=\overline{T}(1\otimes 1)$.
Hence, $2\overline{T}(e)^2 = \overline{T}(e) u + u\overline{T}(e)$, and so $\overline{T}(e) u + u\overline{T}(e)$ commutes with $\overline{T}(e)$. Observe that this can be equivalently stated as  
that $\overline{T}(e)^2 $ commutes with $u$.  That is, $\overline{S}(e)$ commutes with $u$ for every idempotent $e\in A_K$. The algebra $A_K\cong M_n(K)$ is linearly spanned by its idempotents (indeed, observe that the matrix unit $e_{ij}$, $i\ne j$, is a difference of two idempotents:  $e_{ij} = (e_{ii} + e_{ij})-e_{ii}$). As $\overline{S}$ is surjective it follows that $u$ is a scalar multiple of $1\otimes 1$. Since $u =\overline{T}(1\otimes 1)=1\otimes T(1)$, this shows that $\alpha=T(1)\in Z$. Of course, $\alpha\ne 0$. %Setting $\beta = \alpha^{-1}\in Z$ it follows 
From 
\eqref{h300} we see that $ \alpha^{-1} S(x) =T(x)$ for all $x\in A$. Hence, \eqref{h30} implies that $J(x) = \alpha^{-1} T(x)$ %satisfies  $$W(xy +yx)=W(x)W(y)+W(y)W(x) \quad\mbox{for all $x,y\in A$.}$$ This means that $W$ 
is a Jordan automorphism. %, and hence  an automorphism or an antiautomorphism. Of course, $J\in M(A)$ since $T\in M(A)$. % If $J$ is an automorphism, then it is inner by the Skolem-Noether Theorem.

 It remains to show that $\alpha ^3=1$. %, or equivalently, $\beta^3=1$. 
 By Lemma \ref{l2}, $J(x)^3=J(x^3) \equiv x^3$, and by our assumption, $T(x)^3 \equiv x^3$.  Hence, $J(x)^3\equiv T(x)^3$. Since 
 $T(x)=\alpha J(x)$ it follows that $(\alpha^3 - 1)J(x)^3\equiv 0$. As we saw at the end of the proof of Theorem \ref{d}, there exist elements in $A$ whose cube does not lies in $[A,A]$.
 Since $J$ is surjective  it follows that $\alpha^3 - 1=0$.
 
 (ii)$\implies$(i).  As
 $T$ lies in  $M(A)$,
 it leaves every ideal of $A$ invariant. We may therefore assume without loss of generality that $A$ is simple. 
 Now, $J\in M(A)$  implies that 
  $J$ is $Z$-linear, and so % We may therefore consider $A$ as a central simple algebra (over $Z$).
 Lemma \ref{l2} shows that $J(x^3)\equiv x^3$. % where $\beta=\alpha^{-1}$. 
 Since $$J(x^3) = J(x) ^3 = \alpha^{-3}T(x)^3 = T(x)^3,$$ this proves (i).
\end{proof}

The author is thankful to Misha Chebotar for suggesting the trick with idempotents after equation \eqref{h3}.
%The next remark is analogous to Remark \ref{rd}. It shows that in Theorem \ref{a} we cannot replace the condition  $T(x)^3- x^3 \in [A,A]$ by the simpler condition  $T(x)^2- x^2 \in [A,A]$.

The purpose of the following  example  is to show  that  the simpler condition  $T(x)^2- x^2 \in [A,A]$ is not characteristic for Jordan automorphisms (compare Example \ref{rd}).

\begin{example}\label{rh}
If $a,b\in A$ are such that $a ^2=ab=ba=b^2=0$, then $T(x) = x + axb-bxa$ satisfies $$T(x)^2 - x^2 = [xa,xb] + [ax,bx]\in [A,A]$$ and $T(1)=1$. However, $T$ is not always a Jordan automorphism.
\end{example}

The next two examples show that Theorem \ref{a} does not hold if char$(F)$ is $2$ or $3$.

\begin{example}\label{eaut1} Let $F$ be a field with char$(F)=2$, let  $A=M_2(F)$, and let $T:A\to A$ be given by $T(x)=x+{\rm tr}(x)1$. It is easy to see that $T$ is neither an automorphism nor an antiautomorphism.
However, 
\begin{align*}T(x)^3 - x^3 &= {\rm tr}(x)\big(x^2 + {\rm tr}(x)x\big) + {\rm tr}(x)^31\\
& =  {\rm tr}(x)  {\det}(x)1 +  {\rm tr}(x)^31 \\&=  \big({\rm tr}(x)  {\det}(x) +  {\rm tr}(x)^3\big)[e_{12},e_{21}]\in [A,A].\end{align*}
Moreover, $T(1)=1$.
Thus, $T$ satisfies condition (i), but does not satisfy condition (ii).
\end{example}

\begin{example}\label{remaut} If $A$ is an algebra over a field $F$ with char$(F)= 3$,
then for all $x,y\in A$ we have
$$(x+y)^3 - x^3 - y^3 = [x,[x,y]] + [y,[y,x]]\in [A,A]$$
(compare Remark \ref{remder}). Taking any $a\in A$ such that $a^3 \in [A,A]$ (say, $a^3=0$) and any linear functional $\varphi$ on $A$, we thus see that the map $T:A\to A$ given by $T(x)=x+\varphi(x)a$ 
satisfies condition (i). However, $T$ does not necessarily satisfy condition (ii).
\end{example}

The next corollary considers general finite-dimensional algebras.

\begin{corollary}\label{ac2}
Let $A$ be a finite-dimensional algebra over a field $F$ with {\rm char}$(F)\ne 2,3$. 
If a linear map $T:A\to A$ satisfies
 $T(x)^3- x^3 \in [A,A]$ for every $x\in A$, then $T(x^4)- T(x)^4\in {\rm rad}(A)$ for every $x\in A$. Moreover, if $A$ is unital and $T(1)=1$, then  $T(x^2)- T(x)^2\in {\rm rad}(A)$ for every $x\in A$.
\end{corollary}

\begin{proof}  Lemma \ref{ldo} implies that  $T({\rm rad}(A))\subseteq {\rm rad}(A)$. Therefore, we can  define $\overline{T}:A/{\rm rad}(A)\to A/{\rm rad}(A)$ by
$$\overline{T}(x+ {\rm rad}(A)) = T(x)+ {\rm rad}(A).$$
Since 
$$\overline{T}(v)^3 - v^3\in [A/{\rm rad}(A),A/{\rm rad}(A)]\quad\mbox{for all $v\in A/{\rm rad}(A)$}$$
it follows from Theorem \ref{a} that there exist a Jordan automorphism $\overline{J}$ of $A/{\rm rad}(A)$ and an element $\overline{\alpha}$ 
 from the center of $A/{\rm rad}(A)$  such that  $\overline{\alpha}^3 =1$ and $\overline{T}(v)=\overline{\alpha}\overline{J}(v)$ for all $v\in  A/{\rm rad}(A)$. Accordingly,
 $$\overline{T}(v^4) = \overline{\alpha}\overline{J}(v^4) = \overline{\alpha}\overline{J}(v)^4 =   \overline{\alpha}^4\overline{J}(v)^4 = \overline{T}(v)^4$$
 for all $v\in  A/{\rm rad}(A)$, which shows that $T(x^4)- T(x)^4\in {\rm rad}(A)$ for all $x\in A$. Finally, if $A$ is unital and $T(1)=1$, then also $\overline{T}(1) = 1$ and hence $\overline{\alpha}=1$. Thus,
 $\overline{T}$ is a Jordan automorphism and so $T(x^2)- T(x)^2\in {\rm rad}(A)$ for every $x\in A$.
\end{proof}

Assuming that the field  $F$ is  perfect, which makes it possible for us to use the Wedderburn Principal Theorem, we obtain a nicer result that is more similar to Theorem \ref{d}.  We will assume for simplicity that  $A$ is unital and $T(1)=1$.

\begin{corollary}\label{ac3}
Let $A$ be a unital finite-dimensional algebra over a perfect field $F$ with {\rm char}$(F)\ne 2,3$. 
If a linear map $T:A\to A$ satisfies $T(1)=1$ and
 $T(x)^3- x^3 \in [A,A]$ for every $x\in A$, then 
 $T$ is the sum of a Jordan endomorphism of $A$ and  a linear map from $A$ to {\rm rad}$(A)$.
\end{corollary}

\begin{proof}
By the Wedderburn Principal Theorem, $A$ contains a subalgebra $S$ (isomorphic to $A/{\rm rad}(A)$) such that $A$ is  the vector space direct sum of $S$ and rad$(A)$. Let $\pi$ be the projection on $S$ along rad$(A)$.  As $T-\pi T$ has image in rad$(A)$ we must only prove that $\pi T$ is a Jordan endomorphism. 
Now, 
 Corollary \ref{ac2} tells us that $\pi(T(x^2)- T(x)^2)=0$, and since 
$\pi$ is an endomorphism this can be written as $(\pi T)(x^2)=(\pi T)(x) ^2$. 
\end{proof}

%One may wonder whether 

We continue with an analog of Corollary \ref{cd2}. 

\begin{corollary}\label{ca}
Let $A$ be a finite-dimensional semisimple algebra over a field $F$ with {\rm char}$(F)\ne 2,3$. Then every local inner automorphism
$T:A\to A$ is a Jordan automorphism belonging to $M(A)$.
\end{corollary}

\begin{proof}Our assumption can be read as that  for each $x\in A$, there is an invertible $a_x\in A$ such that $T(x)=a_xx a_x^{-1}$. Hence, 
$$T(x)^3 - x^3 = [a_x x ^3, a_x^{-1}]\in [A,A].$$ As $T(1)=1$, Theorem \ref{a} gives the desired conclusion.
\end{proof}

There are many algebras in which every local automorphism is an automorphism. However,
 the matrix algebra $M_n(F)$ is not one of them. Indeed, any matrix  $x \in M_n(F)$ is similar to  its transpose  
$x ^t$, so 
$x\mapsto x^t$ is an example of a local inner automorphism which is not an automorphism but an antiautomorphism. This explains why Jordan automorphisms appear in the conclusion of Corollary \ref{ca}. Moreover, it indicates that 
 in just about any reasonable class of finite-dimensional algebras, the question whether local Jordan automorphisms are Jordan automorphisms is more natural than the usual question whether  local automorphisms are  automorphisms.

Our last theorem gives an answer to the question just  raised. In its proof we will use  the following elementary lemma. Actually, we will need only its special  case where each $m_i=3$. The general form, however, may be useful elsewhere.

\begin{lemma} 
\label{lv}
Let $F$ be an infinite field, let $V$and $W$  be  vector spaces over  $F$, and let $n$ and $m_1,\dots,m_n$ be positive integers. Suppose that $m_i$-linear maps
$f_i:V^{m_i}\to W$ are such that for each $x\in V$, at least one of the elements $f_1(x,\dots,x), \dots, f_n(x,\dots, x)$ is $0$. Then there exists an $i\in \{1,\dots, n\}$ such that
$f_i(x,\dots,x)=0$ for every $x\in V$.
\end{lemma} 

\begin{proof}
Suppose the lemma is not true. Then for each  $i\in \{1,\dots, n\}$ there exists an $x_i\in V$ such that $f_i(x_i,\dots,x_i)\ne 0.$ 
Choose a linear functional $\tau_i$ on $W$ such that $\tau_i(f_i(x_i,\dots,x_i))\ne 0$. For any $z_1,\dots, z_n\in F$, define
$$p_i(z_1,\dots,z_n)=\tau_i(f_i(z_1 x_1 + \dots + z_n x_n, \dots, z_1 x_1 + \dots + z_n x_n)).$$
Note that the assumption of the lemma implies that for every $(z_1,\dots,z_n)\in F ^n$, there is an $i\in \{1,\dots, n\}$ such that
$p_i(z_1,\dots,z_n)=0$. 
Consequently,
\begin{equation}\label{loca}p_1(z_1,\dots,z_n)\cdots p_n(z_1,\dots,z_n) =0\end{equation}
for all  $(z_1,\dots,z_n)\in F ^n$.

 Since $f_i$ is $m_i$-linear, we may consider  $p_i$ as a polynomial in $z_1,\dots, z_n$.
Its  coefficient at $z_i ^{m_i}$ is  $\tau_i(f_i(x_i,\dots,x_i))$, so $p_i\ne 0$.
 The product $p_1\cdots p_n$ is therefore a nonzero polynomial too. However, \eqref{loca} shows that 
 this polynomial vanishes at every  $(z_1,\dots,z_n)\in F ^n$, which 
   contradicts the assumption that $F$ is infinite. %With this contradiction the lemma is proved.
\end{proof}

Lemma \ref{lv} will make it possible for us to use Theorem \ref{a} in the proof of the following theorem, provided of course that the field $F$ is infinite. The  theorem also holds for finite fields, but for them we will have to use a different method which is 
 more similar to standard methods for tackling local automorphisms.

\begin{theorem}\label{a2}
Let $A$ be a finite-dimensional simple algebra over a field $F$ with {\rm char}$(F)\ne 2,3$. Then every local Jordan automorphism
$T:A\to A$ is a Jordan automorphism.
\end{theorem}

\begin{proof}
We consider separately  two cases.
\smallskip 

{\bf Case 1:  $F$ is infinite}. 
The center $Z$ of $A$ is a finite extension of $F$, so there are only finitely many $F$-linear automorphisms of $Z$. Denote them by  $\sigma_1,\dots,\sigma_n$. The restriction of every Jordan automorphism 
of $A$ to $Z$ is therefore one of the $\sigma_i$'s. 

For each $x\in A$, there is a Jordan automorphism $T_x$ of $A$ such that $T(x)=T_x(x)$. Therefore, $A$ is the union of its subsets 
$$A_i=\{x\in A\,|\,  \left.T_x\right|_Z =\sigma_i\},\,\,\,i=1,\dots,n.$$
For every $i$ such that $A_i\ne \emptyset$ choose an $x_i\in A_i$ and set $T_i = T_{x_i}$. 

Take  $x\in A_i$. Then $T_i^{-1} T_x$ fixes elements from $Z$ and is  therefore a 
$Z$-linear Jordan automorphism. Hence,
$$(T_i^{-1} T)(x) ^3 - x^3  = (T_i^{-1} T_x)(x) ^3 - x^3 =  (T_i^{-1} T_x)(x ^3) - x^3\in [A,A]$$
by Lemma \ref{l2}.
%, by the Skolem-Noether Theorem, there exists an invertible $a_x\in A$ such that $(T_i^{-1} T_x)(y)=a_x y a_x^{-1} $ for every $y\in A$. In particular, $(T_i^{-1} T )(x)=a_x x a_x^{-1} $, and hence
%$(T_i^{-1} T)(x) ^3 - x^3 =  [a_x x ^3, a_x^{-1}]\in [A,A]$. 
This shows that  $A_i$ is a subset of the set
$$B_i=\{x\in A\,|\, (T_i^{-1} T )(x) ^3 - x^3 \in [A,A]\}.$$
Therefore, $A=\bigcup_{i=1} ^n B_i$. 

Define $f_i:A ^3\to A/[A,A]$ by $$f_i(x,y,z) = (T_i^{-1} T)(x) (T_i^{-1} T)(y) (T_i^{-1} T)(z) - xyz + [A,A].$$
Note that $x\in B_i$ if and only if $f_i(x,x,x)=0$. The conditions of Lemma \ref{lv} are therefore satisfied, and so there exists an   $i\in \{1,\dots,n\}$ such that 
$f_i(x,x,x)=0$ for every $x\in A$. That is, $(T_i^{-1} T )(x) ^3 - x^3 \in [A,A]$ for every $x\in A$. Theorem \ref{a} shows that $T_i^{-1} T$ is a Jordan  automorphism ($\alpha=1$ since $(T_i^{-1} T)(1)=1$). But then the same holds for $T=T_i(T_i^{-1} T)$. This completes the proof for this case.

\smallskip

{\bf Case 2: $F$ is finite}. Without loss of generality we may assume that $F$ is equal to its prime subfield $\mathbb F_p$. The center $Z$ of $A$ is  then   the finite field $\mathbb F_{p^{n}}$, and, by Wedderburn's theorems, $A$ can be identified with the matrix algebra $M_s(Z)$ for some $s\ge 1$.

Let $\sigma$ denote the  Frobenius automorphism of $Z$.
It is well known that the only automorphisms of $Z$ are $\sigma^{i}$, $i=0,1,\dots,n-1$. Also, it is well known that  $Z$ contains a normal basis, i.e.,  there exists an $a\in Z$ such that
the elements $\sigma^i(a)=a^{p^i}$, $i=0,1,\dots,n-1$, form a basis of $Z$ over $F$. 

The restriction of every Jordan automorphism of $A$ to $Z$ is an automorphism of $Z$. Therefore, for
each $z\in Z$ there exists an $i$ such that $T(z)=\sigma^i(z)$. Let $k$ be such that $T(a)= \sigma^k(a)$. Take $j\ge 1$. Then there exist $l,m$ such that  $$T(\sigma^j(a))= \sigma^l(a)$$ and 
 $$T(a-\sigma^j(a))= \sigma^m(a-\sigma^j(a)) =\sigma^m(a) - \sigma^{j+m}(a).$$
 On the other hand, 
 $$T(a-\sigma^j(a))=  T(a)  - T(\sigma^j(a)) = \sigma^k(a) -  \sigma^l(a).$$
Hence,
$$\sigma^m(a) - \sigma^{j+m}(a) =  \sigma^k(a) -  \sigma^l(a).$$
Since $\sigma^i(a)$, $i=0,1,\dots,n-1$, are linearly independent and  {\rm char}$(F)\ne 2$, it follows that $m=k$ and $j+m\equiv l$ (mod\,$n$).  This means that   $$T(\sigma^j(a))= \sigma^k(\sigma^j(a))\quad\mbox{for every $j\ge 1$.}$$ Consequently,
 $\left.T\right|_Z = \sigma^k$. We can extend $\sigma^k$ to an automorphism $S$ of $A$ in the obvious way, that is, $S((z_{ij}))=(\sigma^k(z_{ij}))$ for every matrix $(z_{ij})\in A$. Note that $S ^{-1}T$ is a local Jordan automorphism of $A$ that acts as the identity on $Z$. Therefore, there is no loss of generality in assuming that 
 $T$ itself is the identity on $Z$.
 Also, we may now assume that $s\ge 2$.
 
 Let $A_0=M_s(F)$. It is obvious that  $\left.T\right|_{A_0}$, the restriction of $T$ to $A_0$, is  an $F$-linear local Jordan homomorphism from $A_0$ to $A$. By \cite[Theorem 2.1]{BS},  $\left.T\right|_{A_0}$ is the sum of a homomorphism $\Phi$ and an antihomomorphism $\Psi$. Suppose $\Phi\ne 0$. From $\Phi(1)= \Phi(e_{11}) + \dots + \Phi(e_{ss})$ (where as above $e_{ii}$ are matrix units)  we see that
 $\Phi(1)$ is an idempotent that can be written as a sum of $s$ mutually orthogonal 
 nonzero idempotents in $A$.  This  implies that $\Phi(1)=1$. Similarly, $\Psi\ne 0$ yields $\Psi(1)=1$. However,
 since $1=T(1)=\Phi(1)+\Psi(1)$, one of $\Phi$ and $\Psi$ is $0$. That is, $\left.T\right|_{A_0}$ is either a homomorphism or an antihomomorphism. Hence, $R:A\to A$ given by $$R\big(\sum z_{ij}e_{ij}\big)=\sum z_{ij}T(e_{ij})\quad\mbox{for all $z_{ij}\in Z$}$$
 is a ($Z$-linear) automorphism or antiautomorphism of $A$. Note that $R^{-1}T$ is a local Jordan automorphism of $A$
 that satisfies $(R^{-1}T)(a_0)= a_0$ for every  $a_0\in A_0$. Therefore, without loss of generality we may assume that $T$  itself is the identity on $A_0$.
 
 The proof will be completed by showing that $T$ is the identity on $A$. 
 
 We start the proof with a general remark. Observe that a nonzero matrix $a\in A$ has rank one if and only if $aAa \subseteq Za$. This shows that every Jordan automorphism of $A$ maps the set of rank one matrices onto itself. The same is then true for every local Jordan automorphism. Thus, $a$ has rank one if and only if $T(a)$ has rank one.
% Every Jordan automorphism $J$ of $A$ is either of the form $J((x_{ij}))= a(\sigma^i(x_{ij}))a^{-1}$ or  $J((x_{ij}))= a(\sigma^i(x_{ji}))a^{-1}$ for some $i$ and an invertible $a\in A$.

Take and fix an arbitrary rank-one idempotent $e$ belonging to $A_0$. For any $z\in Z$, we have $$T(ze)=T_{ze}(ze)=
 T_{ze}(z)T_{ze}(e).$$
  Note that 
 $T_{ze}(z)\in Z$ and $T_{ze}(e)$ is a rank-one idempotent.
 Thus, for every $z\in Z$ there exist a $z'\in Z$ and a rank-one idempotent $e_z\in A$ such that $T(ze)= z' e_z$. Similarly,
 $T(z(1-e)) = z'' f_z$ where $z''\in Z$ and $f_z$ is an idempotent (of rank $s-1$, but we will not need this).
 Hence,
$$z = T(z)  = T(ze) + T(z(1-e))   = z' e_z + z'' f_z.$$
This implies that
 $z''(e_zf_z - f_z e_z)=0$. Since  $z''\ne 0 $  whenever $z\ne 0$  it follows that $e_z$ and $f_z$ commute for every $z\in Z$.
Further, squaring $ z=z' e_z + z'' f_z$  we obtain
$$z^2 = z'^2 e_z + z''^2 f_z + 2z'z''  f_ze_z= z'(z-z'' f_z) + f_z(z''^2 + 2z'z'' e_z),$$
which gives
$$(z-z')z = f_z(z''^2 + 2z'z'' e_z - z'z'').$$
Since $f_z$ is a nontrivial idempotent  it follows that $z=z'$ or $z=0$. We have thereby proved that for every $z\in Z$, there is a rank-one idempotent $e_z\in A$  such that $T(ze) = ze_z$. We may assume that $e_0=e$.

  Since $e\in A_0$, $T(e)=e$.
For every $z\in Z$ we thus have
$$(z+1)e_{z+1} = T((z+1)e) = T(ze) + T(e) = ze_z + e.$$
 By squaring we obtain
 $$(z+1)^2 e_{z+1}= z^2 e_z + z(e_z e + ee_z) + e.$$
 On the other hand,
  $$(z+1)^2 e_{z+1} = (z+1) (ze_z + e).$$
  Comparing the last two identities we see that $z( e_{z} - e)^2=0$. This shows that  $a_z= e_{z} - e$ has square $0$ for every $z\in Z$. Since $e_z = e + a_z$ is an idempotent it follows that   $a_z=ea_z +a_ze$. Consequently,
  $ea_z e =0$ and hence
  $$a_z = (1-e)a_ze + ea_z(1-e).$$
  Suppose $a_z\ne 0$. Then at least one of $ (1-e)a_ze $ and  $ea_z(1-e)$, let us say  the latter one, is nonzero.    
  Take any $u\in  (1-e)A_0 e$ such that $ (1-e)za_ze \ne u$. Then $ze-u\in Ae$ has rank one, but
  $$T(ze -u) = z(e+a_z) - u=  ze +\big( (1-e)za_ze - u\big)+e  za_z(1-e) $$
  has rank at least $2$ since it is a sum of an element from $Ze$, a    nonzero element from $(1-e)Ae$, and a nonzero element from $eA(1-e)$. 
  This contradiction shows that $a_z=0$.
 Thus, $T(ze)= ze$ for every $z\in Z$ and every rank one idempotent $e\in A_0$.
 
 Since $e_{ii}$ and $e_{ii} + e_{ij}$, where $i\ne j$, are rank-one idempotents belonging to $A_0$, it follows that 
 $T(ze_{ii})= ze_{ii}$ and
 $T(ze_{ij})= ze_{ij}$ for all $z\in Z$ and all $1\le i,j\le s$. But then $T(a)=a$ for every $a\in A$.
 \end{proof}

%\begin{remark}
%Let $K$ be a finite extension of a field $F$. Then every ($F$-linear) local automorphism of $K$ is an automorphism.   Indeed, there are only finitely many automorphisms $\sigma_1,\dots,\sigma_n$. Thus, if $\beta$ is a local automorphism of $K$, then $$K=\bigcup_{i=1}^n \{x\in K\,|\, \beta(x)=\sigma_i(x)\}.$$ However,
 %a vector space over an infinite field cannot be  the union of finitely many subspaces (see, e.g., 
 %\cite[Corollary 13]{LM}). We may therefore assume that $F$ is a finite field. If $|F|=p^n$, then
 %$\sigma_i = \varphi^{i-1}$, $i=1,\dots,n$, where $\varphi$ is the Frobenius automorphism of $K$.
%\end{remark}

%We conclude with a trivial corollary to Theorem \ref{a2}, which may nevertheless be of some interest in light of Corollary \ref{ca}.

%\begin{corollary}\label{ca2}Let $A$ be a finite-dimensional simple algebra over a field $F$ with {\rm char}$(F)\ne 2,3$. Then every local  automorphism
%$T:A\to A$ is  either an  automorphism or an antiautomorphism.
%\end{corollary}

\end{document}